\documentclass[12pt,a4paper]{article}
\usepackage{amsfonts}
\usepackage{amsmath}
\usepackage{amssymb}
\usepackage{fullpage}
\usepackage{amsthm}
\usepackage{graphicx}

\newtheorem{theorem}{Theorem}
\newtheorem{proposition}[theorem]{Proposition}
\newtheorem{lemma}[theorem]{Lemma} 
\newtheorem{corollary}[theorem]{Corollary}

\newcommand{\nn}{{\mathbb N}}
\newcommand{\nnp}{{\nn_+}}
\newcommand{\zz}{{\mathbb Z}}

\newcommand{\calb}{{\mathcal{B}}}
\newcommand{\calh}{{\mathcal{H}}}
\newcommand{\calp}{{\mathcal{P}}}
\newcommand{\calq}{{\mathcal{Q}}}

\title{Iteration of functions and contractibility of acyclic 2-complexes} 
\author{Ian J. Leary\thanks{Partially supported by a Research Fellowship 
from the Leverhulme Trust.  This work was started at MSRI, Berkeley, 
where research is supported by 
the National Science Foundation under Grant No.~DMS-1440140.}
}

\date{\today}

\begin{document} 

\maketitle

\begin{abstract} 
We show that there can be no algorithm to decide whether infinite
recursively described acyclic aspherical 2-complexes are contractible.
We construct such a complex that is contractible if and only if the
Collatz conjecture holds.  
\end{abstract}

\section{Introduction} 

The existence of an algorithm to determine which finite 2-dimensional
simplicial complexes are contractible is a well-known open problem.
There are good algorithms to compute the homology of a finite complex,
and so the problem quickly reduces to the case of {\sl acyclic}
complexes, i.e., those having the same homology as a point.  This
problem can be stated as a problem about finite presentations of
groups: is there an algorithm to decide which finite balanced
presentations of perfect groups are trivial?  The problem was stated
in this form and attributed to Magnus in the first edition of the
Kourovka notebook~\cite[1.12]{kourovka}, and appears as problem~(FP1)
in~\cite{nylist}.  A presentation is {\sl balanced} if it has the same
numbers of generators and relators.  The conditions that the group be
perfect and that the presentation be balanced are equivalent to the
corresponding 2-complex being acyclic.

If a finite presentation presents the trivial group, then a 
systematic search will eventually find a proof of this 
fact.  Thus there is a partial algorithm that will verify that 
a finite simplicial complex (of any dimension) is contractible, 
but that will fail to halt in general.  

In general there is no algorithm to decide whether a finite presentation 
describes the trivial group.  This implies that there can be no algorithm
to decide whether a finite 2-dimensional simplicial complex is 
simply-connected.  However, this does not imply that there can 
be no algorithm to decide whether a finite 2-complex is contractible, 
because all known families of problematic group 
presentations have many more relators than generators.  
The corresponding presentation 2-complexes cannot be contractible
because they have non-trivial second homology.  

We have nothing to say about this well-known problem, but instead we
consider an infinite analogue.  In contrast to the finite case, there
is no algorithm for computing the homology of recursively described
infinite complexes.  For example, it has been known since the work of
Collatz in the 1930's that it is difficult to decide such questions as
whether a recursively described graph is connected~\cite{lagarias}.
(We will also justify this assertion in Theorem~\ref{thm:conn} below.)
For this reason we consider only acyclic complexes.  Here is our main
result.  

\begin{theorem} \label{thm:main}
There is no algorithm to decide whether an 
infinite, recursively described, aspherical, acyclic presentation 
2-complex is contractible.  

Moreover, there is no partial algorithm to find all the 
contractible complexes in this class, and there is no 
partial algorithm to find all the non-contractible complexes.  
\end{theorem} 

Usually decidability results in group theory rely on the 
existence of a non-recursive, recursively enumerable set.  
Instead we rely on the dynamics of functions on the set 
$\nnp$ of strictly positive integers.  We associate a 2-complex 
$P(f)$ to each $f:\nnp\rightarrow\nnp$ whose contractibility
is controlled by the orbits of $f$.  For such a function $f$,
define a group presentation $\calp(f)$ in which both the
generators and relators are indexed by $\nnp$: 
\[\calp(f):= \langle a_i,\,\,i\in \nnp\,\,:\,\,
a_{f(i)}^{-1}a_ia_{f(i)}=a_i^2\,\rangle.\]
Now let $P(f)$ be the presentation 2-complex associated to $\calp(f)$, 
so that the 1- and 2-cells of $P(f)$ are indexed by $\nnp$.  Recall 
also that a {\sl forward orbit} for $f$ is a subset of $\nnp$ 
of the form $\{i,f(i),f^2(i),\ldots\}$ for some $i\in \nnp$.  

\begin{lemma}\label{lemma:main}
  For each $f$, the 2-complex $P(f)$ is both acyclic and aspherical.
  The following are equivalent.
  \begin{itemize}
  \item{} $P(f)$ is contractible;
  \item{} $\calp(f)$ presents the trivial group; 
  \item{} every forward orbit of $f$ is eventually periodic of
    period at most 3.
  \end{itemize}
\end{lemma}

There is a striking corollary concerning the Collatz function.  
Recall that the Collatz or hailstone function 
$C:\nnp\rightarrow \nnp$ is defined by
\[C(n)= \begin{cases}
  3n+1&\hbox{for $n$ odd},\\
  n/2&\hbox{for $n$ even}.\\
    \end{cases}\] 
The {\sl Collatz conjecture} states that every forward orbit of $C$
contains 1.  

\begin{corollary}\label{corollary:collatz}
  $P(C)$ is contractible if and only if the Collatz conjecture holds.
\end{corollary}

This corollary follows easily from Lemma~\ref{lemma:main}.  The
proof of Theorem~\ref{thm:main} depends also on a result of
Kurtz and Simon concerning the decidability of properties
of functions~\cite{kurtz-simon}.

\section{Proofs} 

The presentation 2-complex associated to a group presentation is a
CW-complex with one 0-cell, with 1-cells in bijective correspondence
with the generators, and 2-cells in bijective correspondence with the
relators.  We shall assume that each relator is a cyclically reduced
word in the generators.  The 1-skeleton of the 2-complex is a rose,
whose fundamental group is naturally identified with the free group on
the set of generators in the given presentation.  An element of this
group describes a based homotopy class of maps from the circle to the
1-skeleton.  The relator corresponding to a 2-cell is used in this way
to describe its attaching map.  For details see~\cite[p.~50]{hatcher}.
A presentation 2-complex is {\sl acyclic} if it has the same homology
as a point, and is {\sl aspherical} if its universal covering space is
contractible, or equivalently if it is an Eilenberg--Mac~Lane space for 
the group presented.  Some authors use the term `aspherical presentation' 
for a more general situtation that arises when some of the relators are 
proper powers~\cite{lynsch}; this will not concern us.  

There is an algorithm to pass from presentation 2-complexes as 
described above to homotopy equivalent simplicial complexes.  Each 
2-cell corresponding to a relator of length $n$ should be viewed
as an $n$-gon; the second barycentric subdivision of the polygonal 
cell complex obtained in this way is a simplicial complex.  Each
petal of the original 1-skeleton rose is triangulated as the 
boundary of a square, and the 2-cell corresponding to a relator 
of length $n$ is built from $12n$ triangles. 

Conversely, if $K$ is a connected 2-dimensional simplicial complex and 
$T$ is a maximal tree in $K$, the quotient space $K/T$ is naturally a 
CW-complex with one 0-cell, and so it may be viewed as a presentation 
complex, with generators corresponding to the edges of $K-T$ and relators
corresponding to the 2-simplices of $K$.  This process too is algorithmic, 
given $K$ and $T$, but when $K$ is infinite there may be no algorithm to 
find a maximal tree~$T$.  

The discussion in the two paragraphs above leads to the following 
proposition.  

\begin{proposition} 
The existence of an algorithm to decide contractibility for finite 
presentation 2-complexes is equivalent to the existence of an algorithm
to decide contractibility for finite 2-dimensional simplicial complexes.  
\end{proposition} 

\begin{proof} 
  Follows from the discussion above, since there is an algorithm to
  find a maximal tree inside a finite connected 2-dimensional
  simplicial complex.  
\end{proof} 

The discussion above also leads to a corollary to our main theorem.  

\begin{corollary} 
For the class of recursively defined acyclic aspherical 2-dimensional 
simplicial complexes, there is no partial algorithm to find all 
contractible complexes in the class, and no partial algorithm to 
find all non-contractible complexes. 
\end{corollary} 
 
\begin{proof} 
Since there is an algorithmic way to pass from a recursively described
presentation 2-complex to a recursively defined 2-dimensional simplicial
complex, a partial algorithm of the type mentioned in the statement 
of the corollary would contradict Theorem~\ref{thm:main}.  
\end{proof} 

Two infinite families of finite group presentations will play a 
role in our proofs, the families $\calb(n)$ and $\calh(n)$ given 
below.  As above, we use the notation $B(n)$ and $H(n)$ for the 
corresponding presentation 2-complexes.  

$$\calb(n):= \langle a_1,\ldots,a_{n+1} \,\,:\,\, a_i^{a_{i+1}}=a_i^2 
\rangle$$ 

$$\calh(n):= \langle a_i, \, i\in \zz/n \,\,:\,\, a_i^{a_{i+1}}=a_i^2
\rangle$$ 

These presentations first appeared as steps in Higman's construction 
of an infinite finitely generated simple group~\cite{higman}.  
To establish properties of $B(n)$, $H(n)$ and of the complexes $P(f)$, 
we will use two well-known propositions.  

\begin{proposition}\label{prop:ggp} 
Let $X$, $Y$ and $Z$ be Eilenberg--Mac~Lane spaces for the 
groups $H$, $K$ and $L$ 
respectively, and let $i:Z\rightarrow X$ and $j:Z\rightarrow Y$ be based 
maps such that the induced maps on fundamental groups are injective: 
$i_*:L\rightarrow H$ and $j_*:L\rightarrow K$.  
\begin{enumerate} 
\item The double mapping cylinder or homotopy colimit 
\[(X\sqcup Z\times I\sqcup Y)/(z,0)=i(z),\, (z,1)=j(z)\] is an 
Eilenberg--Mac~Lane space for the free product with amalgamation 
$H*_LK=\langle H,K\colon i_*(l)=j_*(l),\,\, l\in L\rangle$. 

\item If moreover $X=Y$ so that $K=H$, then the homotopy colimit 
of the smaller diagram with two arrows and just two spaces, 
$(X\sqcup Z\times I)/ (z,0)=i(z),\, (z,1)=j(z)$, is an
Eilenberg--Mac~Lane space for the HNN-extension 
$H*_L = \langle H,t\colon ti_*(l)t^{-1}=j_*(l),\,\,l\in L\rangle$.  
\end{enumerate} 
\end{proposition} 

\begin{proof} 
Both assertions are special cases of a theorem concerning arbitrary 
graphs of groups; see for example~\cite[1.B.11]{hatcher} although 
note that the directed graph used in~\cite[1.B]{hatcher} is the 
barycentric subdivision (with edges oriented from the larger edge 
midpoint to the vertex) of the graph usually used to index a graph
of groups.
\end{proof} 

\begin{proposition}\label{prop:fpa}
With hypotheses as in part (1) of Proposition~\ref{prop:ggp}, if in addition
the maps $i:Z\rightarrow X$ and $j:Z\rightarrow Y$ are isomorphisms 
from $Z$ to subcomplexes of $X$ and $Y$, then the coproduct
$X\sqcup Y/i(z)=j(z)$ is also an Eilenberg--Mac~Lane space for 
$H*_LK$.  
\end{proposition} 

\begin{proof} 
The inclusion of a subcomplex in a CW-complex is a cofibration, and hence 
the pair $(X,i(Z))$ is homotopy equivalent to the pair $(M_i,Z)$, where $M_i$ 
denotes the mapping cylinder of $i:Z\rightarrow X$ and similarly $(X,j(Z))$ 
is homotopy equivalent to the pair $(M_j,Z)$.  The claim follows.  
\end{proof}

\begin{corollary} \label{cor:bsn}
The presentation 2-complex $B(n)$ for $\calb(n)$ is aspherical, and 
each of the generators $a_i$ represents an element of its fundamental
group, $\pi_1(B(n))$, of infinite order.   
The presentation obtained by adding 
the relation $a_{n+1}=1$ to $\calb(n)$ presents the trivial group.  
The subgroup of $\pi_1(B(n))$ generated by $a_1$ and $a_{n+1}$ is 
free of rank two provided that $n\geq 2$.  
\end{corollary}  

\begin{proof} 
The group presented by $\calb(1)$ is the Baumslag-Solitar group 
$BS(1,2)$, which is an HNN-extension of the 
infinite cyclic group $\langle a_1\rangle$ with stable letter 
$a_2$.  The asphericity claim for $n=1$ follows from part (2) of 
Proposition~\ref{prop:ggp}.  Given the relation $a_2=1$, 
the relator in $\calb(1)$ 
reduces to $a_1=a_1^2$, which immediately implies $a_1=1$.  This 
completes the proof when $n=1$.  

All of the claims except that final one are proved by induction
using Proposition~\ref{prop:fpa}, 
since a complex isomorphic to $B(n+1)$ can be obtained by 
identifying the circle labelled $a_{n+1}$ in $B(n)$ with 
the circle labelled $a_1$ in a second copy of $B(1)$.  

For the final claim, note that the cyclic groups $\langle a_1\rangle$ and 
$\langle a_{n+1}\rangle$ have trivial intersection inside $\pi_1(B(n))$ 
for each~$n\geq 1$.  It follows that $\langle a_1,a_{n+2}\rangle$ is 
a free group by applying the Normal Form Theorem for free products 
with amalgamation to the given decomposition of 
$\pi_1(B(n+1))$~\cite[IV.2.6]{lynsch}.  
\end{proof} 

\begin{corollary} \label{cor:higman}
The presentation 2-complex $H(n)$ for $\calh(n)$ is aspherical for all $n$,
and is contractible if and only if $n\leq 3$.  For $n\geq 4$, each 
$a_i$ generates an infinite cyclic subgroup of $\pi_1(H(n))$.  
\end{corollary} 

\begin{proof} 
It can be checked readily with a computer algebra
package that $\calh(n)$ presents the trivial group for $n\leq 3$, 
or see~\cite[pp63--64]{higman} for a direct proof.  
For $n\geq 4$, the complex $H(n)$ can be obtained from a copy of
$B(n-2)$ and a copy of $B(2)$ by identifying the 2-petalled
rose labelled by $a_1$ and $a_{n-1}$ inside $B(n-2)$ with the
2-petalled rose labelled by $a_3$ and $a_1$ respectively inside
$B(2)$.  The remaining claims concerning $H(n)$ now follow
from Proposition~\ref{prop:fpa} and Corollary~\ref{cor:bsn}.
\end{proof}  

We are now ready to start the proof of Lemma~\ref{lemma:main}.  

\begin{proof}
It is immediate that each $P(f)$ is acyclic, and so we 
concentrate on the homotopy groups.  We view the 1-~and~2-cells
of $P(f)$ as being indexed by $\nnp$, so that the generator 
$a_i$ corresponds to the loop around the $i$th edge, and the 
$i$th 2-cell is the 2-cell corresponding to the relation 
$a_{f(i)}^{-1}a_ia_{f(i)}=a_i^2$.  

First, consider the case of the function $f(i):=i+1$.  In this 
case, the 2-cells indexed by $1,\ldots,n$ and the 1-cells indexed 
by $1,\ldots,n+1$ form a subcomplex isomorphic to $B(n)$, and the 
whole complex is the ascending union of these subcomplexes.  It 
follows that in this case $P(f)$ is aspherical as claimed.  
To see that $P(f)$ is not contractible, let  
$r:\nnp \rightarrow \zz/4$ be the map defined by $r(i):=[i]:=i+4\zz$.  
This induces a cellular map $P(f)\rightarrow H(4)$ which 
is surjective on fundamental groups.  Since the subgroup $\langle 
a_{r(i)}\rangle$ is infinite, so is the subgroup $\langle a_i\rangle$ 
of the group presented by $\calp(f)$.  

Now we move on to the general case.  We fix a function $f:\nnp\rightarrow 
\nnp$, and so we write $\calp$ instead of $\calp(f)$ and $P$ in place of 
$P(f)$, and prove the 
claim by induction on certain subcomplexes of $P$.  The boundary
of the 2-cell indexed by $i\in\nnp$ meets only the 1-cells $a_i$, 
$a_{f(i)}$ and the 0-cell.  Hence if $S\subseteq \nnp$ is any subset
such that $f(S)\subseteq S$, the cells indexed by $S$ (together with 
the 0-cell) form a subcomplex of~$P$.  Denote this subcomplex by $P(S)$.  
For each $i\in\nnp$, write $F(i)$ for the forward orbit of $i$: 
$F(i):=\{f^n(i)\,:\,n\geq 0\}$, and note that $f(F(i))\subseteq F(i)$.  

As an inductive hypothesis, suppose that we have an $f$-closed subset
$S\subseteq \nnp$ so that $P(S)$ is aspherical, and that $P(S)$ 
satisfies two other properties analogous to those claimed for $P(f)$,   
as described below.  Firstly, we suppose that $P(S)$ is contractible 
if and only if every forward orbit of $f|_S$ is eventually periodic 
of period at most~3.  Secondly, we suppose that for each $i\in S$, 
the subgroup of $\pi_1(P(S))$ generated by $a_i$ is either infinite 
or trivial, and is trivial iff the forward orbit $F(i)$ is eventually 
periodic of period at most~3.  

Let $i$ be the least element of $\nnp-S$.  There are two main  
cases to consider, depending whether $F(i)\cap S$ is empty or 
non-empty.  

If $F(i)\cap S\neq \emptyset$, let $n$ be minimal
so that $f^n(i)\in S$.  In this case, $P(F(i)\cup S)$ is isomorphic
to the union of $P(S)$ and the subspace $X$ consisting of the 0-cell, 
the 1-cells indexed by $\{i,f(i),\ldots,f^{n}(i)\}$ and the 2-cells 
indexed by the set $\{i,f(i),\ldots,f^{n-1}(i)\}$.  The subspace~$X$ 
is isomorphic to $B(n)$.  The intersection of $P(S)$ and $X$ is the 
circle consisting of the 0-cell and the 1-cell indexed by 
$f^n(i)$.  If the forward orbit of $i$ is eventually periodic of 
period at most~3, then so is the forward orbit of $f^n(i)$, and so 
the circle indexed by $f^n(i)$ is contractible in $P(S)$, and 
attaching a copy of $B(n)$ to this circle does not change the 
homotopy type, so $P(F(i)\cup S)$ is homotopy equivalent to 
$P(S)$.  If on the other hand the forward orbit of $i$ is 
infinite, or eventually periodic of period greater than~3, then 
the circle indexed by $f^n(i)$ represents an element of infinite
order in the fundamental groups of both $P(S)$ and $X\cong 
B(n)$.  The inductive hypothesis passes to $P(S\cup F(i))
= P(S)\cup_{f^n(i)}X$.  

If $F(i)\cap S=\emptyset$, then $P(F(i)\cup S) = 
P(F(i))\vee P(S)$, the 1-point union of $P(F(i))$ 
and $P(S)$, and $\pi_1(P(F(i)\cup S)$ 
is just the free product $\pi_1(P(F(i)))*\pi_1(P(S))$.  
Thus our inductive hypothesis will hold for $F(i)\cup S$
provided that we can show that it holds for $F(i)$.  This splits 
into three further subcases.  If $F(i)$ is 
infinite, then $P(F(i))$ is isomorphic to the complex discussed 
in the first paragraph, and the claims hold.  If $f$ is periodic of
period $n$ when restricted to $F(i)$, then $P(F(i))$ is isomorphic
to the complex $\calh(n)$, and the inductive hypotheses hold by 
Corollary~\ref{cor:higman}.  If neither of these cases holds, we 
take an intermediate step between $S$ and $S\cup F(i)$.  
Pick the least $n>0$ so that there exists $m>n$ with $f^m(i)=f^n(i)$, 
and let $j:=f^n(i)$.  We have that $f$ is periodic on $F(j)$ and 
$F(j)\cap S=\emptyset$, so by the cases already covered we deduce 
that the inductive hypotheses hold for $S':=S\cup F(j)$.  Note also
that $F(i)\cap S'=F(j)\neq\emptyset$, so getting from $S'$ to $S'\cup 
F(i)=S\cup F(i)$ reduces to the other main case considered in the 
previous paragraph.  
\end{proof} 

Next we prove Corollary~\ref{corollary:collatz}.  

\begin{proof}
By Lemma~\ref{lemma:main}, 
it suffices to show that any forward orbit of the Collatz function~$C$
that does not eventually join the periodic orbit $1,4,2,\ldots$ is 
either infinite or has eventual period greater than~3.  So let
$n_1,n_2\ldots,n_l,\ldots$ be a periodic orbit of period $l$, 
and choose the starting point so that $n_1<n_i$ for $2\leq i\leq l$.  
This implies that $n_1$ is odd, and that $n_2=C(n_1)=3n_1+1$.  Since 
$3n_1+1>2n_1$, we see that $C(n_2)>n_1$ and so $l>2$.  If $l=3$, it 
must be that $C(C(n_2))=n_1$, and hence $3n_1+1=4n_1$, which implies
that $n_1=1$.  
\end{proof}  

A generalized Collatz function is a function $g:\nnp\rightarrow \nnp$ 
such that there exists an integer $m>0$ and rationals $a_i,b_i$ for 
$0\leq i<m$ so that whenever $x$ is congruent to $i$ modulo $m$, 
$g(x)=a_ix+b_i$.  In particular, the function $C$ can be written in 
this way for $m=2$, $a_0=1/2$, $b_0=0$, $a_1=3$, $b_1=1$.  
To prove Theorem~\ref{thm:main}, we quote a theorem of Kurtz and 
Simon~\cite{kurtz-simon}, that strengthens a well known result 
due to Conway~\cite{conway,lagarias}.  They define GCP to be the problem 
of deciding, for each generalized Collatz function $g$, whether every
forward orbit of $g$ contains~1.  They
show~\cite[thrm.~3]{kurtz-simon} that GCP is $\Pi^0_2$-complete; 
and hence in particular there can be no partial algorithm that 
can identify either the generalized Collatz functions that satisfy
GCP or the ones that do not.  

Given a function $f:\nnp\rightarrow \nnp$, we define a new function 
$\widehat f:\nnp\rightarrow \nnp$ as follows.  Firstly, let $\phi:
\nnp\rightarrow \nnp\times \zz/4$ be the function $\phi(n)=
(\lfloor (n+3)/4\rfloor,[n])$, where as before we use $[n]$ to 
denote the class $n+4\zz\in \zz/4$, and $\lfloor q \rfloor$ is the 
greatest integer less than or equal to $q$.  Note 
that $\phi$ is a bijection and that
both $\phi$ and $\phi^{-1}$ are easily computable.  Now define 
$\widetilde{f}:\nnp\times \zz/4\rightarrow \nnp\times \zz/4$ by 
\[\widetilde{f} (m,[i]):=\begin{cases} 
(f(m),[i+1])&\text{for}\,\,\, m\neq 1,\\
(1,[i])&\text{for}\,\,\, m=1.\\
\end{cases}\] 
Finally, define $\widehat{f}:=\phi^{-1}\circ\widetilde{f}\circ\phi$.  

\begin{lemma}\label{lemma:fhat}
The following are equivalent, for any function
$f:\nnp\rightarrow\nnp$. 
\begin{itemize} 
\item{} Every forward orbit of $f$ contains 1; 
\item{} Every forward orbit of $\widehat f$ is eventually periodic 
of period at most~3.
\end{itemize} 
\end{lemma} 

\begin{proof} 
Since $\phi$ is a bijection, the dynamics of $\widehat{f}$ and 
$\widetilde{f}$ are identical.  It is clear that for any $i,x$,
the forward orbit of $(x,[i])$ under $\widetilde{f}$ will be 
eventually constant (i.e., eventually periodic of period~1) if
the forward orbit of $x$ under $f$ contains~1.  
Similarly, if the forward orbit of $x$ under $f$ does not 
contain~1, then for each $n\geq 0$ we see that 
$\widetilde{f}^n(x,[i])=(f^n(x),[i+n])$, 
and so the forward orbit of $(x,[i])$ will be either infinite
or eventually periodic of period divisible by~4.  
\end{proof} 

We are now ready to prove Theorem~\ref{thm:main}.  

\begin{proof} 
Let $g$ be an arbitrary generalized Collatz function, and consider
the question of whether $P(\widehat{g})$ is contractible.  By 
Lemma~\ref{lemma:main}, this is equivalent to the question of 
whether every forward orbit of $\widehat{g}$ is eventually periodic
of period at most three.  By Lemma~\ref{lemma:fhat}, this is 
equivalent to the question of whether every forward orbit of $g$ contains~1. 
But Kurtz and Simon showed that this question is $\Pi^0_2$-complete, 
which implies the claim.
\end{proof}

\section{Closing remarks} 

Similar but simpler techniques can be used to prove undecidability
results for the homology of infinite complexes, and we give two 
examples below.  

\begin{theorem}\label{thm:conn}
There is no algorithm to decide whether a recursively described
graph is connected.  Moreover, there is no partial algorithm to 
find the connected ones, and no partial algorithm to find the 
disconnected ones.  
\end{theorem} 

\begin{proof} 
For a function $f:\nnp\rightarrow \nnp$, define a graph $\Gamma(f)$ 
with vertex and edge set indexed by $\nnp$, where the vertices of the 
edge $e_i$ are $v_{i+1}$ and $v_{f(i+1)}$.  This graph is connected if and
only if every forward orbit of $f$ contains~1.  

By the Kurtz-Simon theorem, the question of whether the graph
$\Gamma(g)$ is connected, for $g$ any generalized Collatz function, is
$\Pi^0_2$-complete, which implies the claim.
\end{proof} 

A similar argument applies to the computation of the first homology 
of recursively described 2-complexes, because the group with 
presentation $\calq(f)$ given by 
\[\calq(f):= \langle a_i,\,\,i\in \nnp\,\,:\,\,
a_1=1,\,\,a_{i+1}=a_{f(i+1)}\rangle\]
is free, and is trivial if and only if every forward 
orbit of $f$ contains~1.  

The top homology group (i.e., $H_1$ for graphs or $H_2$ for 2-complexes) 
is different however: 

\begin{proposition} \label{prop:toph} For each $n$, 
there is a partial algorithm that identifies the recursively described
$n$-dimensional complexes with non-zero $n$th homology group.  
\end{proposition} 

\begin{proof} 
Fix an integer parameter $m>0$.  Now compute the boundaries of the 
first $m$ of the $n$-cells, expressed as formal sums of $n-1$-cells. 
Use standard techniques of linear algebra to decide whether these 
$m$ elements are linearly dependent; if so then $H_n\neq 0$.  If not, 
increase $m$ and repeat.  
\end{proof} 

The functions that we have considered are of course far simpler than
arbitrary recursive functions.  If $g$ is either a generalized Collatz
function or $g=\widehat{f}$ for some generalized Collatz function $f$,
then the sets $g^{-1}(i)$ are of bounded size and can be easily
computed.  From this one sees that the cellular cochain complexes for
$P(g)$, $Q(g)$ and $\Gamma(g)$ are also recursively described, where 
$Q(g)$ denotes the presentation 2-complex for $\calq(g)$.  The
argument used in Proposition~\ref{prop:toph}, when applied
$H^0(\Gamma(g))$ for any $g$ for which the cochain complex for
$\Gamma(g)$ is recursively described, gives that there is a partial
algorithm that will identify when $\Gamma(g)$ has a \emph{finite}
connected component.

\leftline{\bf Author's address:}

\obeylines

\smallskip
{\tt i.j.leary@soton.ac.uk}

\smallskip
School of Mathematical Sciences, University of Southampton, 
Southampton, SO17 1BJ

\end{document}